\documentclass[a4paper,UKenglish,cleveref, autoref, thm-restate]{lipics-v2021}

\usepackage{amsmath}
\usepackage{amssymb}
\usepackage{amsfonts}
\usepackage{amsthm}
\usepackage{mathtools}
\usepackage{graphicx}
\usepackage{booktabs}
\usepackage{hhline}
\usepackage[colorinlistoftodos]{todonotes}
\usepackage{soul}

\newtheoremstyle{TheoremNum}
	{\topsep}{\topsep}{}{0pt}{\sffamily}{. }{5pt}             
                              {$\vartriangleright$ \thmname {\bfseries #1} \thmnote {\bfseries #3}}
\theoremstyle{TheoremNum}
\newtheorem{thmn}{Theorem}
\newtheorem{lemman}{Lemma}
\newtheorem{obsn}{Observation}

\usepackage{tikz}
\usepackage{caption, subcaption}
\usepackage{cleveref}

\usepackage{graphics}
\usepackage{array}

\newcommand{\R}{\mathbb{R}}

\newcommand{\cC}{\mathcal{C}}

\newcommand{\cTC}{\mathcal{TC}}
\newcommand{\cOr}{\mathcal{OR}}
\newcommand{\cTB}{\mathcal{TB}}

\newcommand{\dpVC}{\textsc{Degree-3 Vertex Cover}}
\newcommand{\dpLD}{\textsc{$L_\cOr$-Distance}}

\newcommand{\leo}[1]{\textcolor{black!50!black}{#1}}

\newcommand{\markj}[1]{\textcolor{black!50!black}{#1}}

\DeclareMathOperator{\head}{head}
\DeclareMathOperator{\tail}{tail}
\DeclareMathOperator{\Gad}{Gad}

\definecolor{lightblue}{RGB}{180,180,255}

\definecolor{lightgreen}{RGB}{180,255,150}

\title{Making a Network Orchard by Adding Leaves}

\author{Leo {van Iersel}}{Delft Institute of Applied Mathematics, Delft University of Technology, The Netherlands}{L.J.J.vanIersel@tudelft.nl}{https://orcid.org/0000-0001-7142-4706}{Research funded in part by Netherlands Organization for Scientific Research (NWO) grants OCENW.KLEIN.125 and OCENW.GROOT.2019.015.}

\author{Mark Jones}{Delft Institute of Applied Mathematics, Delft University of Technology, The Netherlands}{M.E.L.Jones@tudelft.nl}{[orcid]}{Research funded by Netherlands Organization for Scientific Research (NWO) grant OCENW.KLEIN.125.}

\author{Esther Julien}{Delft Institute of Applied Mathematics, Delft University of Technology, The Netherlands}{E.A.T.Julien@tudelft.nl}{https://orcid.org/0000-0002-7337-1086}{Research funded by Netherlands Organization for Scientific Research (NWO) grant OCENW.GROOT.2019.015.}

\author{Yukihiro Murakami \footnote{Corresponding author}}{Delft Institute of Applied Mathematics, Delft University of Technology, The Netherlands}{Y.Murakami@tudelft.nl}{https://orcid.org/0000-0003-1355-5884}{}

\authorrunning{L. van Iersel, M. Jones, E. Julien, and Y. Murakami} 

\Copyright{Leo van Iersel, Mark Jones, Esther Julien, and Yukihiro Murakami} 



\relatedversion{} 

\supplementdetails[subcategory={}]{Source Code}{https://github.com/estherjulien/OrchardProximity}

\acknowledgements{}

\nolinenumbers 

\EventEditors{}
\EventNoEds{0}
\EventLongTitle{}
\EventShortTitle{}
\EventAcronym{}
\EventYear{}
\EventDate{}
\EventLocation{}
\EventLogo{}
\SeriesVolume{}
\ArticleNo{}

\begin{document}

\maketitle

\begin{abstract}
Phylogenetic networks are used to represent the evolutionary history of species. Recently, the new class of orchard networks was introduced, which were later \leo{shown to be interpretable} as trees with additional horizontal arcs. This makes the network class ideal for capturing evolutionary histories that involve horizontal gene transfers. \leo{Here, we} study the minimum number of additional leaves needed to make \leo{a} network orchard. We demonstrate that computing this proximity measure for a given network is NP-hard and describe a tight upper bound. We also give an equivalent measure based on vertex labellings to construct a mixed integer linear programming formulation. Our experimental results, which include both real-world and synthetic data, illustrate the effectiveness of our implementation.
\end{abstract}

\section{Introduction}
\label{sec:intro}
Phylogenetic trees are used to represent the evolutionary history of species. 
While they are effective for illustrating speciation events through vertical descent, they are insufficient in representing more intricate evolutionary processes.
Reticulate (net-like) events such as hybridization and horizontal gene transfer (HGT) can give rise to signals that cannot be represented on a single tree~\cite{goulet2017hybridization,wickell2020evolutionary}.
In light of this, phylogenetic networks have gained increasing attention due to their capability in elucidating reticulate evolutionary processes. 

Phylogenetic networks are often categorized into different classes based on their topological features. 
These are often motivated computationally, but some classes are also defined based on their biological relevance~\cite{pardi2015reconstructible}.
Classical examples of network classes involve the \emph{tree-child networks}~\cite{cardona2008comparison} and the \emph{tree-based networks}~\cite{francis2015phylogenetic}.
Roughly speaking, tree-child networks are those where every vertex has passed on a gene via vertical descent to an extant species, and tree-based networks are those obtainable from a tree by adding so-called \emph{linking arcs} between tree arcs.
Recent developments have culminated in the introduction of \emph{orchard networks}, which lie 
 -- inclusion-wise -- between the two aforementioned network classes~\cite{janssen2021cherry,erdHos2019class}.
The class has shown to be both algorithmically attractive and biologically relevant; they are defined as networks that can be reduced to a single leaf by a series of so-called \emph{cherry-picking operations}, and they were shown to be networks that can be obtained by adding horizontal arcs to trees (where the tree is drawn with the root at the top and arcs pointing downwards)~\cite{van2022orchard}. Such horizontal arcs can be used to model HGT events, making orchard networks \leo{especially} apt in representing evolutionary scenarios where every reticulate event is \leo{a horizontal transfer}.
Orchard networks have also been characterized statically based on so-called \emph{cherry covers}~\cite{van2021unifying}.

When considering a non-orchard network, a natural question arises: how many additional leaves are required to transform the network into one that is orchard? 
From a biological standpoint, this question can be interpreted as asking how many extinct species or unsampled taxa need to be introduced into the network to yield a scenario where every reticulation represents an HGT event. 
Given that HGT is the primary driver of reticulate evolution in bacteria~\cite{gyles2014horizontally}, this is an essential inquiry. 
We provide a network of a few fungi species in \Cref{fig:fungi_added_leaf}, which 
requires one additional leaf to make it orchard. 
Formally speaking, the problem of computing this leaf addition measure is as follows.
\begin{figure}
    \centering
    \includegraphics{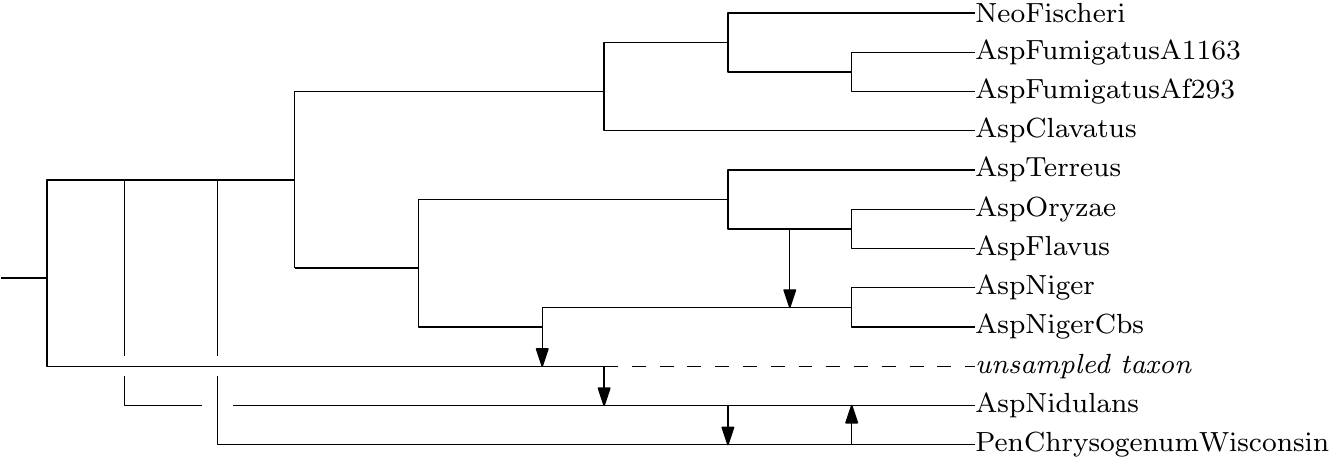}
    \caption{A network on 11 different \leo{taxa} (excluding \emph{unsampled taxon}) of fungi \leo{including 5 reticulations, which is part of a larger network} from~\cite{szollHosi2015genome}. The directed arcs in the figure are linking arcs, \leo{which represent gene transfer highways}.
    \leo{In order to make all linking arcs horizontal}, we require an additional leaf (\textit{unsampled taxon}) to represent the evolutionary history. 
    }
    \label{fig:fungi_added_leaf}
\end{figure}

\medskip
\noindent
\fbox{
\parbox{0.8\linewidth}{
{\sc \dpLD{} (Decision)}\\
{\bf Input:} A network $N$ on a set of taxa $X$ and a natural number $k$.\\
{\bf Decide:} Can $N$ be made orchard with at most $k$ leaf additions?}
}
\medskip

In related research, the leaf addition measure has been investigated for other network classes. 
It has been shown that tracking down the minimum leaf additions to make a network tree-based can be done in polynomial time~\cite{francis2018new}. In the same paper, it was shown that the leaf addition measure was equivalent to two other proximity measures, namely those based on spanning trees and disjoint path partitions.
The same question was posed for the unrooted variant (where the arcs of the network are undirected), for which the problem turned out to be NP-complete~\cite{fischer2020tree}. A total of eight proximity measures were introduced in this latter paper, including those based on edge additions and rearrangement moves.
Instead of considering leaf additions, some manuscripts have even considered leaf deletions (in general, vertex deletions) as proximity measures for the class of so-called \emph{edge-based networks}~\cite{fischer2022far}.
Finally for orchard networks, a recent bachelor's thesis compared how the leaf addition proximity measure differs in general to another proximity measure based on arc deletions~\cite{susanna2022making}. 

In this paper, we show that the leaf addition proximity measure can be computed in polynomial time for the class of tree-child networks, and we give a more efficient algorithm for computing the measure for tree-based networks.
we show that \dpLD{} is NP-complete by a polynomial time reduction from \dpVC{}. 
To model the problem as a \emph{mixed integer linear program} (MILP), we consider a reformulation of the leaf addition measure in terms of 
vertex labellings.
Orchard networks are known to be trees with horizontal arcs; roughly speaking, this means we can label the vertices of an orchard network so that every vertex of indegree-$2$ has exactly one incoming arc whose end-vertices have the same labels. 
The reformulated measure, called the
\emph{vertical arcs} proximity measure, counts -- over all possible vertex labellings (defined formally in \Cref{subsec:ProximityMeasures_H}) -- the minimum number of indegree-$2$ vertices with only non-horizontal arcs. 
Our experimental results are promising, as the real world cases are solved in a fraction of a second. Furthermore, the model also behaves well for synthetic data.


The structure of the paper is as follows.
In \Cref{sec:preliminaries}, we provide all necessary definitions and characterizations of orchard networks and tree-based networks. In \Cref{sec:Proximity}, we formally introduce the leaf addition measure for the classes of tree-child, orchard, and tree-based networks. 
In \Cref{sec:Hardness} we show that \dpLD{} is NP-complete (\Cref{thm:L_Or=Hard}).
In \Cref{sec:Bound}, we give a sharp upper bound for the leaf addition proximity measure.
In \Cref{sec:ILP} we give a reformulation of the leaf addition measure to describe the MILP to solve~\dpLD{}, and in \Cref{sec:experiments}, experimental results are shown for the MILP, applied to real and simulated networks. In \Cref{sec:discussion}, we give a brief discussion of our results and discuss potential future research directions. We include proofs for select results in \Cref{sec:appendix}.

\section{Preliminaries}
\label{sec:preliminaries}

A \emph{binary directed phylogenetic network} on a non-empty set~$X$ is a 
directed acyclic graph with
\begin{itemize}
    \item a single \emph{root} of indegree-0 and outdegree-1;
    \item \emph{tree vertices} of indegree-1 and outdegree-2;
    \item \emph{reticulations} of indegree-2 and outdegree-1;
    \item \emph{leaves} of indegree-1 and outdegree-0, that are labelled 
    bijectively by elements of~$X$.
\end{itemize}

For the sake of brevity, we shall refer to binary directed phylogenetic 
networks simply as \emph{networks}. Throughout the paper, assume that~$N$ is a 
network on some non-empty set~$X$ where~$|X| = n$, unless stated otherwise.
Networks without reticulations are called 
\emph{trees}. Tree vertices and reticulations may sometimes collectively be 
referred to as \emph{internal vertices}.

The arc~$uv$ of a network is a \emph{root arc} if~$u$ is the root of the 
network.
An arc~$uv$ of a network is a \emph{reticulation arc} if~$v$ is a reticulation, 
and a \emph{tree arc} otherwise. We say that a vertex~$u$ is a \emph{parent} of 
another vertex~$v$ if~$uv$ is an arc of the network; in such instances we 
call~$v$ a \emph{child} of~$u$. Also, we say that~$u$ and~$v$ are the 
\emph{tail} and the \emph{head} of the arc~$uv$, respectively. In other words, 
we may rewrite arcs as~$uv = \tail(uv)\head(uv)$. The \emph{neighbours} of~$v$ 
refer to the set of vertices that are parents or children of~$v$.
We also say that vertices~$u$ and~$v$ are \emph{siblings} if they share the same parent.

In what follows, we shall define graph operations based on vertex and arc 
deletions. To make sure resulting graphs remain networks, we follow-up every 
graph operation with a \emph{cleaning up} process. Formally, we \emph{clean up} 
a network by applying the following until none is applicable.
\begin{itemize}
    \item Suppress an indegree-1 outdegree-1 vertex (e.g., if $uv$ and $vw$ are 
    arcs where~$v$ is an indegree-1 outdegree-1 vertex, we suppress~$v$ by 
    deleting the vertex~$v$ and adding an arc~$uw$).
    \item Replace parallel arcs by a single arc (e.g., if~$uv$ is an arc twice 
    in a network, 
    delete one of the arcs $uv$).
\end{itemize}
We observe that deleting a non-reticulation arc and cleaning up results in a 
graph containing two indegree-0 vertices. On the other hand, deleting a 
reticulation arc and cleaning up results in a network. Therefore, we shall 
use arc deletions to mean reticulation arc deletions.

\subsection{Tree-Child Networks}

A network is \emph{tree-child} if every non-leaf vertex has a child that is a 
tree vertex or a leaf. 
We call an internal 
vertex of a network an \emph{omnian} if all of its children are 
reticulations~\cite{jetten2016nonbinary}. It follows from definition that a 
network is tree-child if and only if it contains no omnians.

\subsection{Orchard Networks} \label{sec:orchard}





To define orchard networks, we must first define cherries and reticulated 
cherries, as well as operations to reduce them. 
See \Cref{fig:EGSequence} for the illustration of the following definitions.
Let~$N$ be a network. 
Two leaves~$x$ and~$y$ of~$N$ form a \emph{cherry} if they are siblings. In such a case, we say that~$N$ \emph{contains} a 
cherry~$(x,y)$ or a cherry~$(y,x)$. Two leaves~$x$ and~$y$ of~$N$ form a 
\emph{reticulated cherry} if the parent~$p_x$ of~$x$ is a reticulation and the 
parent of~$y$ is also a parent of~$p_x$. In such a case, we say that~$N$ 
\emph{contains} a reticulated cherry~$(x,y)$. \emph{Reducing the cherry~$(x,y)$ from~$N$} is 
the process of deleting the leaf~$x$ and cleaning up. \emph{Reducing the 
reticulated cherry~$(x,y)$ from~$N$} is the process of deleting the arc from 
the parent of~$y$ to the parent of~$x$ and cleaning up. In both cases, we use~$N(x,y)$ to denote the resulting network.

A network~$N$ is \emph{orchard} if there is a sequence~$S = (x_1,y_1)(x_2,y_2)\ldots(x_k,y_k)$
such that~$NS$ is a network 
on a single leaf~$y_k$. It has been shown that the order in which (reticulated) cherries are reduced does not matter~\cite{erdHos2019class,janssen2021cherry}.
Apart from this recursive definition, orchard networks have been characterized based on cherry covers (arc decompositions)~\cite{van2021unifying} and vertex 
labellings~\cite{van2022orchard}. We include both characterizations here.

\begin{figure}
    \centering
    \includegraphics{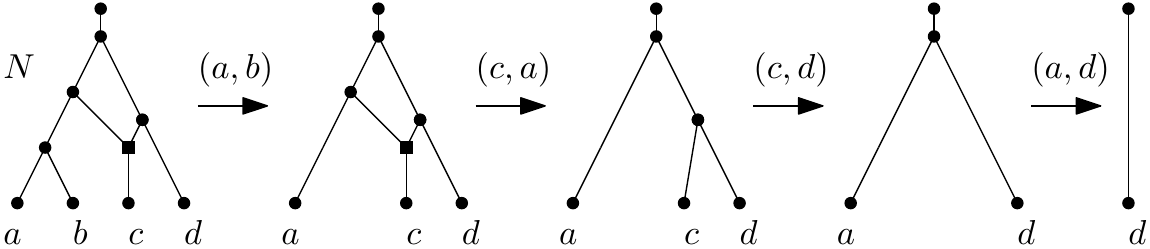}
    \caption{An example of an orchard network~$N$ that is reduced by a sequence~$(a,b)(c,a)(c,d)(a,d)$.
    The network~$N$ contains a cherry~$(a,b)$ and a reticulated cherry~$(c,d)$.
    Subsequent networks are those obtained by a single cherry picking reduction from the previous network. For example, the second network~$N(a,b)$ is obtained from~$N$ by removing the leaf~$b$ and cleaning up.}
    \label{fig:EGSequence}
\end{figure}

\subparagraph{Cherry covers (see \cite{van2021unifying} for more details):} 
A \emph{cherry shape} is a subgraph on three distinct vertices~$x, y, p$ with 
arcs~$px$ and~$py$. 
The \emph{internal vertex} of a cherry shape is $p$, and 
the \emph{endpoints} are~$x$ and~$y$. A \emph{reticulated cherry shape} is a 
subgraph on four distinct vertices~$x, y, p_x, p_y$ with arcs~$p_xx, p_yp_x, 
p_yy$, such that~$p_x$ is a  reticulation in the network. The internal vertices 
of a reticulated cherry shape are~$p_x$ and~$p_y$, and the endpoints are~$x$ 
and~$y$. 
The \emph{middle arc} of a 
reticulated cherry shape is
$p_yp_x$. 
We will often 
refer to cherry shapes and the reticulated cherry shapes by their arcs (e.g., 
we would denote the above cherry shape~$\{p_xx, p_yy\}$ and the reticulated 
cherry shape~$\{p_xx, p_yp_x, p_yy\}$). We say that an arc~$uv$ is covered by 
a cherry or reticulated cherry shape~$C$ if~$uv \in C$. 
A \emph{cherry cover} of a binary network is a set $P$ of cherry shapes and reticulated cherry shapes, such that each arc except for the root arc is covered exactly once by~$P$.
In general, a network can have more than one cherry cover.

We define the \emph{cherry cover auxiliary graph}~$G=(V,A)$ of a cherry cover as 
follows. 
For all shapes~$B \in P$, we have~$v_B \in V$. A shape~$B \in P$ is 
\emph{directly above} another shape~$C \in P$ if an internal vertex of~$C$ is 
an endpoint of~$B$. Then,~$v_Bv_C \in A$. (adapted from \cite[Definition 
2.13]{van2021unifying}). 
We say that a cherry cover is \emph{cyclic} if its auxiliary graph has a cycle. We call it \emph{acyclic} otherwise.
See \cref{fig:cherry_cover_example} for an 
illustration of a cyclic and acyclic cherry cover. 

\begin{figure}[h]
    \centering
        \subfloat[Network $N_1$]{\includegraphics[height=4cm]{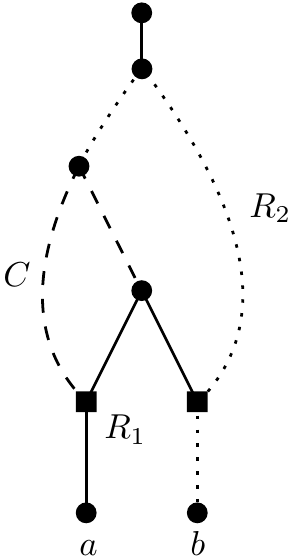} \label{subfig:cc_a}}
        \hspace{2mm}
        \subfloat[Cherry cover aux. graph of $N_1$]{\includegraphics[height=1.2cm]{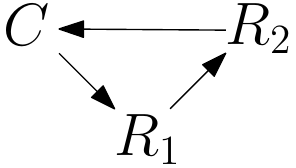} \label{subfig:cc_b}}
        \hspace{5mm}
        \subfloat[Network $N_2$]{\includegraphics[height=4cm]{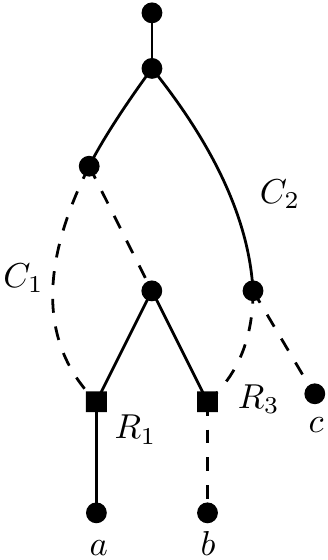} \label{subfig:cc_c}}
        \hspace{2mm}
        \subfloat[Cherry cover aux. graph of $N_2$]{\includegraphics[height=1.2cm]{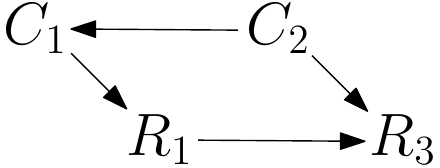} \label{subfig:cc_d}}
    \caption{A cherry cover example. 
    (a) A network~$N_1$ on~$\{a, b\}$ with a cherry cover~$\{C,R_1,R_2\}$.
    (b) The (cyclic) auxiliary graph of~$N_1$ based on the cherry cover of (a). 
    (c) The network~$N_2$ obtained from~$N_1$ by adding a leaf $c$, with a cherry cover~$\{C_1,C_2,R_1,R_3\}$
    (d) The (acyclic) auxiliary graph of~$N_2$ based on the cherry cover of (d).}
    \label{fig:cherry_cover_example}
\end{figure}

\begin{theorem}[Theorem 4.3 of 
	\cite{van2021unifying}]\label{thm:AcyclicCherryCover}
	A network~$N$ is orchard if and only it has an acyclic cherry cover.
\end{theorem}

\subparagraph{Non-Temporal Labellings:}\label{par:nontemporallabelling}

Let~$N$ be a network with vertex set~$V(N)$. A \emph{non-temporal 
labelling}\footnote{This is named in contrast to \emph{temporal 
representations} of~\cite{baroni2006hybrids}. There, it was required for the 
endpoints of every reticulation arc to have the same label.} 
of~$N$ is a labelling~$t:V(N)\rightarrow\R$ such that
\begin{itemize}
	\item for all arcs $uv$, $t(u)\le t(v)$ and equality is allowed only if~$v$ 
	is a reticulation;
	\item for each internal vertex~$u$, there is a child~$v$ of~$u$ such 
	that~$t(u) < t(v)$;
	\item for each reticulation~$r$ with parents~$u$ and~$v$, at most one 
	of~$t(u)=t(r)$ or~$t(v)=t(r)$ holds.
\end{itemize}
Observe that every network (orchard or not) admits a non-temporal labelling by 
labelling each vertex by its longest distance from the root (assuming each arc is of weight 1). 

Under non-temporal labellings, we call an arc \emph{horizontal} if its 
endpoints have the same label; we call an arc \emph{vertical} otherwise. By 
definition, only reticulation arcs can be horizontal. We say that a 
non-temporal labelling is an \emph{HGT-consistent labelling} if every 
reticulation is incident to exactly one incoming horizontal arc. We recall the 
following key result.

\begin{theorem}[Theorem 1 of~\cite{van2022orchard}]\label{thm:OrchIFFHori}
	A network is orchard if and only if it admits an HGT-consistent labelling.
\end{theorem}

\subsection{Tree-Based Networks}

A network~$N$ is \emph{tree-based} with \emph{base tree}~$T$ if it can be 
obtained from~$T$ in the following steps.
\begin{enumerate}
	\item Replace some arcs of~$T$ by paths, whose internal vertices we call 
	\emph{attachment points}; each attachment point is of indegree-1 and 
	outdegree-1.
	\item Place arcs between attachment points, called \emph{linking arcs}, so 
	that the graph contains no vertices of total degree greater than $3$, and 
	so that it remains acyclic.
	\item Clean up.
\end{enumerate}

The relation between the classes of tree-child, orchard, and tree-based 
networks can be stated as follows.

\begin{lemma}[\cite{janssen2021cherry} and Corollary 1 of 
\cite{van2022orchard}]\label{lem:NetworkClassContainment}
	If a network is tree-child, then it is orchard. If a network is orchard, 
	then it is tree-based.
\end{lemma}

We include here a static characterization of tree-based networks based on an 
arc partition, called \emph{maximum zig-zag 
trails}~\cite{hayamizu2021structure, zhang2016tree}. 
Let~$N$ be a network. A \emph{zig-zag trail} of length~$k$ is a sequence~$(a_1,a_2,\ldots, a_k)$ of arcs where~$k\ge 1$, where either~$\tail(a_i) = \tail(a_{i+1})$ 
or~$\head(a_i)=\head(a_{i+1})$ holds for~$i\in [k-1] = \{1,2,\ldots, k-1\}$. We 
call a zig-zag trail~$Z$ \emph{maximal} if there is no zig-zag trail that 
contains~$Z$ as a subsequence. Depending on the nature of~$\tail(a_1)$ 
and~$\tail(a_k)$, we have four possible maximal zig-zag trails. 
\begin{itemize}
	\item \emph{Crowns}: $k\ge 4$ is even and $\tail(a_1) = \tail(a_k)$ or 
	$\head(a_1) 
	= \head(a_k)$.
	\item \emph{M-fences}: $k\ge 2$ is even, it is not a crown, and $\tail(a_i)$ is a tree vertex for every~$i\in[k]$.
	\item \emph{N-fences}: $k\ge 1$ is odd and $\tail(a_1)$ or~$\tail(a_k)$, but not both, is 
	a reticulation. By reordering the arcs, assume henceforth that~$\tail(a_1)$ is a reticulation and~$\tail(a_k)$ a tree vertex.
	\item \emph{W-fences}: $k\ge 2$ is even and both~$\tail(a_1)$ 
	and~$\tail(a_k)$ are 
	reticulations.
\end{itemize}
We call a set~$S$ of maximal zig-zag trails a \emph{zig-zag decomposition} 
of~$N$ if the elements of~$S$ partition all arcs, except for the root arc, 
of~$N$.

\begin{lemma}[Corollary 4.6 of \cite{hayamizu2021structure}]\label{lem:TB=NoW}
    Let~$N$ be a network. Then~$N$ is tree-based if and only if it has no 
    W-fences.
\end{lemma}

\begin{theorem}[Theorem 4.2 of 
\cite{hayamizu2021structure}]\label{thm:UniqueMaxZigZag}
	Any network~$N$ has a unique zig-zag decomposition.
\end{theorem}

\begin{theorem}[Theorem 3.3 of \cite{van2021unifying}]\label{thm:TB=ChCover}
    Let~$N$ be a network. Then~$N$ is tree-based if and only if it has a cherry cover.
\end{theorem}

\section{Leaf Addition Proximity Measure}
\label{sec:Proximity}
Let~$N$ be a network on~$X$. \emph{Adding a leaf~$x\notin X$ to an arc~$e$ 
of~$N$} is the process of adding a labelled vertex~$x$, subdividing the arc~$e$ 
by a vertex~$w$ (if~$e=uv$ then we delete the arc~$uv$, add the vertex~$w$, and 
add arcs~$uw$ and~$wv$), and adding an arc~$wx$. We denote the resulting 
network by~$N+(e,x)$. When the arc~$e$ in the above is irrelevant or clear, we simply 
call 
this process \emph{adding a leaf~$x$ to~$N$}, and denote the resulting network 
by~$N+x$.

In this section, $\cC$ 
will be used to denote a network class. In particular, we shall use~$\cTC, 
\cOr$, and~$\cTB$ to denote the classes of tree-child networks, orchard 
networks, and tree-based networks, respectively.
Let~$L_{\cC}(N)$ denote the minimum number of 
leaf additions required to make the network~$N$ a member of~$\cC$. 
We first show that computing~$L_\cTC(N)$ and~$L_\cTB(N)$ can be done in polynomial time. 

\begin{lemma}\label{lem:L_TC=Omnian}
    Let~$N$ be a network. Then $L_{\cTC}(N)$ is equal to the number of omnians. Moreover,~$N$ can be made tree-child by adding a leaf to exactly one outgoing arc of each omnian. 
\end{lemma}

\begin{lemma}\label{lem:L_TCisPoly}
    Let~$N$ be a network. Then~$L_{\cTC}(N)$ can be computed in~$O(|N|)$ time.
\end{lemma}

It has been shown already that~$L_{\cTB}(N)$ can be computed in~$O(|N|^{3/2})$ time 
where~$|N|$ is the number of vertices in~$N$~\cite{francis2018new}. We show that this can in fact be computed in~$O(|N|)$ time.

\begin{lemma} \label{lem:L_TB=W-fences}
    Let $N$ be a network. Then $L_{\cTB}(N)$ is equal to the number of $W$-fences. Moreover, $N$ is tree-based by adding a leaf to any arc in each $W$-fence in $N$.
\end{lemma}

\begin{lemma}\label{lem:L_TB=polynomial}
	Let~$N$ be a network. Then $L_\cTB(N)$ can be computed in~$O(|N|)$ time. 
\end{lemma}

Interestingly, computing~$L_\cOr(N)$ proves to be a difficult problem, although the leaf addition proximity measure is easy to compute for its neighbouring network classes. We prove the following in~\Cref{sec:Hardness}.
\begin{thmn}[\ref{thm:L_Or=Hard}]
	Let~$N$ be a network. Computing $L_\cOr(N)$ is NP-hard. 
\end{thmn}
We also include the following theorem which states that when considering leaf addition proximity measures for orchard networks, it suffices to consider leaf additions to reticulation arcs. We shall henceforth assume that all leaf additions are on reticulation arcs.
\begin{theorem}[Theorem 4.1 of \cite{susanna2022making}]\label{thm:OnlyAddToRetArcs}
    A network~$N$ is orchard if and only if the network obtained by adding a leaf to a tree arc of~$N$ is orchard.
\end{theorem}
The rest of the paper will now focus on the problem of computing~$L_\cOr(N)$.

\section{Hardness Proof}
\label{sec:Hardness}

In this section, we show that computing~$L_\cOr(N)$ is NP-hard by reducing from  
degree-3 vertex cover. 

\medskip
\noindent
\fbox{\parbox{0.8\linewidth}{
{\sc \dpVC{} (Decision)}\\
{\bf Input:} A 3-regular graph $G = (V, E)$ and a natural number $k$.\\
{\bf Decide:} Does G have a vertex cover with at most $k$ vertices?}}
\medskip

\medskip
\noindent
\fbox{
\parbox{0.8\linewidth}{
{\sc \dpLD{} (Decision)}\\
{\bf Input:} A network $N$ on a set of taxa $X$ and a natural number $k$.\\
{\bf Decide:} Can $N$ be made orchard with at most $k$ leaf additions?}
}
\medskip

\begin{figure}
    \centering    
    \includegraphics[width=\columnwidth]{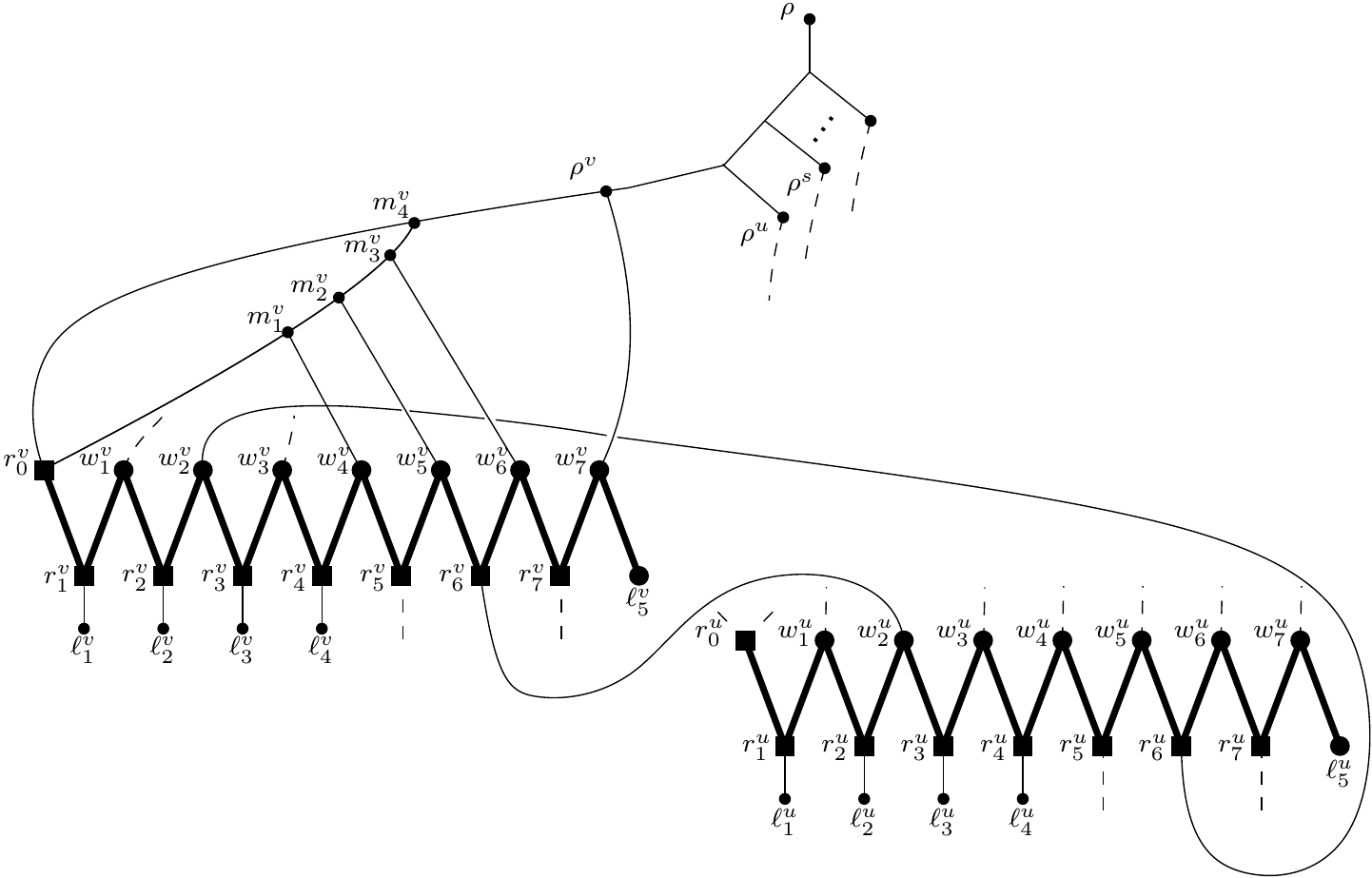}
    \caption{Sketch of the network~$N_G$ for the case when~$G$ contains an edge~$uv$.} 
    \label{fig:VCtoLORDreduction}
\end{figure}

\medskip
We now describe the reduction from \dpVC{} to \dpLD{}. For a graph~$G$, let~$V(G)$ and~$E(G)$ be its vertex and edge sets, respectively. Given an instance~$(G,k)$ of \dpVC, construct an instance~$(N_G,k)$ of \dpLD{} as follows (see \Cref{fig:VCtoLORDreduction}):

\begin{enumerate}
    \item For each vertex~$v$ in $V(G)$, construct a gadget~$\Gad(v)$ as described below.
    In what follows, vertices of the form $\ell_i^v$ are leaves, vertices $r_i^v$ are reticulations, and vertices $w_i^v$, $m_i^v$ and $\rho^v$ are tree vertices.
    
    The key structure in~$\Gad(v)$ is an N-fence with $15$ arcs, starting with the arc 
    $\boldsymbol{r_0^vr_1^v}$, then followed by arcs of the form $\boldsymbol{w_i^vr_i^v, w_i^vr_{i+1}^v}$ for each $i\in[6]$, and finally the arcs $\boldsymbol{w_7^vr_7^v,w_7^v\ell_5^v}$. This set of arcs, in bold type, is called the \emph{principal part} of~$\Gad(v)$. 
    In addition,the reticulations $r_1^v, r_2^v,r_3^v,r_4^v$ have leaf children $\ell_1^v,\ell_2^v, \ell_3^v, \ell_4^v$ respectively.
    
    Above the principal part of~$\Gad(v)$, add a set of tree vertices $m_1^v,m_2^v,m_3^v,m_4^v,\rho^v$ with the following children: $m_1^v$ has children $r_0^v$ and $w_4^v$, $m_2^v$ has children $m_1^v$ and $w_5^v$, $m_3^v$ has children $m_2^v$ and $w_6^v$, $m_4^v$ has children $m_3^v$ and $r_0^v$, and $\rho^v$ has children $m_4^v$ and $w_7^v$ (see \Cref{fig:VCtoLORDreduction}).
    
    This completes the construction of~$\Gad(v)$. Note that so far, the vertices $w_1^v,w_2^v,w_3^v$ have no incoming arcs, and $r_5^v,r_6^v,r_7^v$ have no outgoing arcs. Such arcs will be added later to connected different gadgets together.

    \item Connect the vertices $\rho^v$ from each~$\Gad(v)$ as follows: take some ordering of the vertices~$\{v_1,\ldots, v_g\}$ of~$G$. Add a vertex~$\rho$ and vertices~$s_i$ for~$i\in [g-1]$. Add arcs~$\rho s_1$ and also arcs from the set~$\{s_is_{i+1}: i\in [g-2]\}$, as well as arcs from the set~$\{s_i\rho^{v_i}: i\in [g-1]\}$, and finally an arc~$s_{g-1}\rho^{v_g}$.
    
    \item Next add arcs between the gadgets corresponding to adjacent vertices in $G$, in the following way:
    for every pair of adjacent vertices $u,v$ in $G$, add an arc connecting one of the vertices $r_5^u,r_6^u,r_7^u$ in $\Gad(u)$ to one of the vertices $w_1^v,w_2^v,w_3^v$ in $\Gad(v)$  (and, symmetrically, an arc connecting one of $r_5^v,r_6^v,r_7^v$ to one of $w_1^u,w_2^u,w_3^u$). The exact choice of vertices connected by an arc does not matter, except that we should ensure each vertex is used by such an arc exactly once.
    Formally: for each vertex $v$ in $G$ with neighbours $a,b,c$, fix two (arbitrary) mappings $\pi_v:\{a,b,c\} \rightarrow \{1,2,3\}$ and $\tau_v:\{a,b,c\}\rightarrow \{5,6,7\}$. Then for each pair of adjacent vertices $u,v$ in $G$, add an arc from $r_{\tau_u(v)}^u$ to $w_{\pi_v(u)}^v$ (and, symmetrically, add an arc from $r_{\tau_v(u)}^v$ to $w_{\pi_u(v)}^u$).

    \item Finally, for each vertex~$v$ in~$G$, label the vertices~$\{\ell_i^v: i\in[5]\}$ in~$\Gad(v)$ by~$\ell_i^v$.
\end{enumerate}

Call the resulting graph~$N_G$; it is easy to see that~$N_G$ is directed and acyclic with a single root~$\rho$. Therefore it is a network on the leaf-set~$\{\ell_i^v:i\in[5] \text{, $v\in V(G)$}\}.$ 
\markj{As the arcs of~$N_G$ are decomposed into M-fences and N-fences, we have the following observation.}

\begin{observation}\label{obs:N_GisTB}
    Let~$G$ be a 3-regular graph and let~$N_G$ be the network obtained by the reduction. Then~$N_G$ is tree-based.
\end{observation}

By \Cref{obs:N_GisTB} and \Cref{thm:TB=ChCover}, we use freely from now on that~$N_G$ has a cherry cover.
Before proving the main result, we require some notation and helper lemmas. 
Let~$N$ be a network and let~$\hat{N_i}$ be an N-fence of~$N$. In what follows, we shall write~$\hat{N_i} := (a^i_1,a^i_2,\ldots, a^i_{k_i})$, and we will let~$c^i_{2j-1}$ denote the child of~$\head(a^i_{2j-1})$ for~$j\in\left[\frac{k_i-1}{2}\right]$.
The first lemma states that although a tree-based network may have non-unique cherry covers, the reticulated cherry shapes that cover arcs of N-fences are fixed.

\begin{lemma} \label{lem:n_fence_unique}
    Let~$N$ be a tree-based network, and let~$\hat{N_1},\hat{N_2},\ldots, \hat{N_n}$ denote the N-fences of~$N$ of length at least 3. 
    Then every cherry cover of~$N$ contains the reticulated cherry shapes~$\{(\head(a^i_{2j-1})c^i_{2j-1}), a^i_{2j},a^i_{2j+1}\}$ for~$i\in[n]$ and~$j\in\left[\frac{k_i-1}{2}\right]$.
\end{lemma}

Note that the principal part of a gadget~$\Gad(v)$ for every~$v\in V(G)$ is an N-fence. Let us denote the principal part of a gadget~$\Gad(v)$ by~$(a^v_1,a^v_2,\ldots,a^v_{15})$ for all~$v\in V(G)$. By \Cref{lem:n_fence_unique}, $a^v_i$ for~$i=2,\ldots, 15$ and~$v\in E(G)$ are covered in the same manner across all possible cherry covers of~$N_G$. Let us denote the reticulated cherry shape that contains~$a^v_i$ and~$a^v_{i+1}$ by~$R^v_{i/2}$ for even $i \in [15]$. 
\Cref{subfig:gadgets_cherry_cover,subfig:gadgets_cherry_cover_graph} show an example of the part of cherry cover auxiliary graph containing~$R^v_{i}$ and~$R^u_i$ for~$i\in[7]$, for some edge~$uv$ in $G$.
Note that the cherry shapes form a cycle.
The next lemma implies that in fact,  such a cycle exists for any edge~$uv$ in $G$.

\begin{figure}
    \centering
    \subfloat[]{\includegraphics[height=4cm]{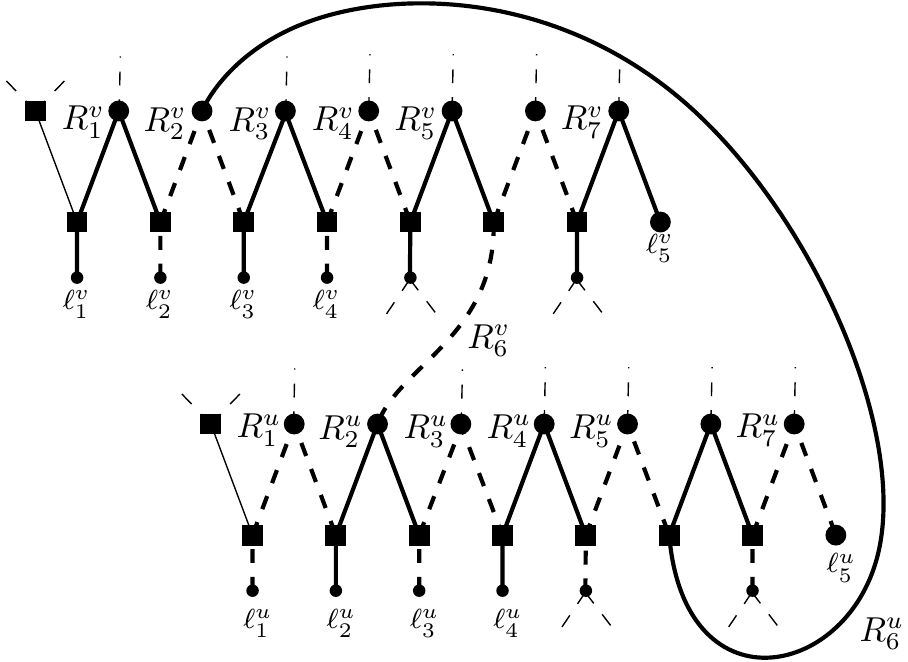} \label{subfig:gadgets_cherry_cover}}
        \subfloat[]{\includegraphics[height=4cm]{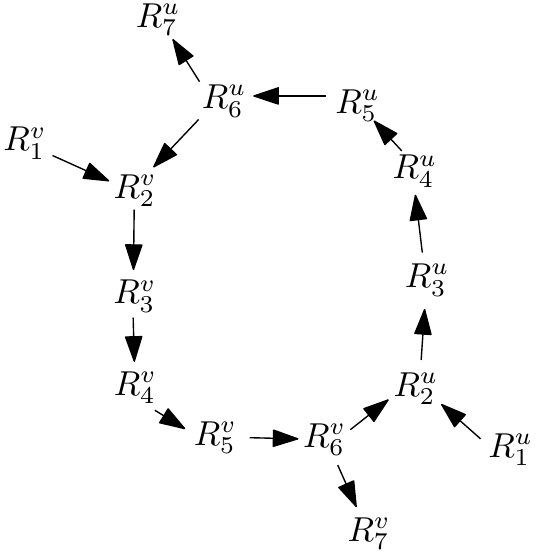}\label{subfig:gadgets_cherry_cover_graph}}
        \\
    \subfloat[]{\includegraphics[height=4cm]{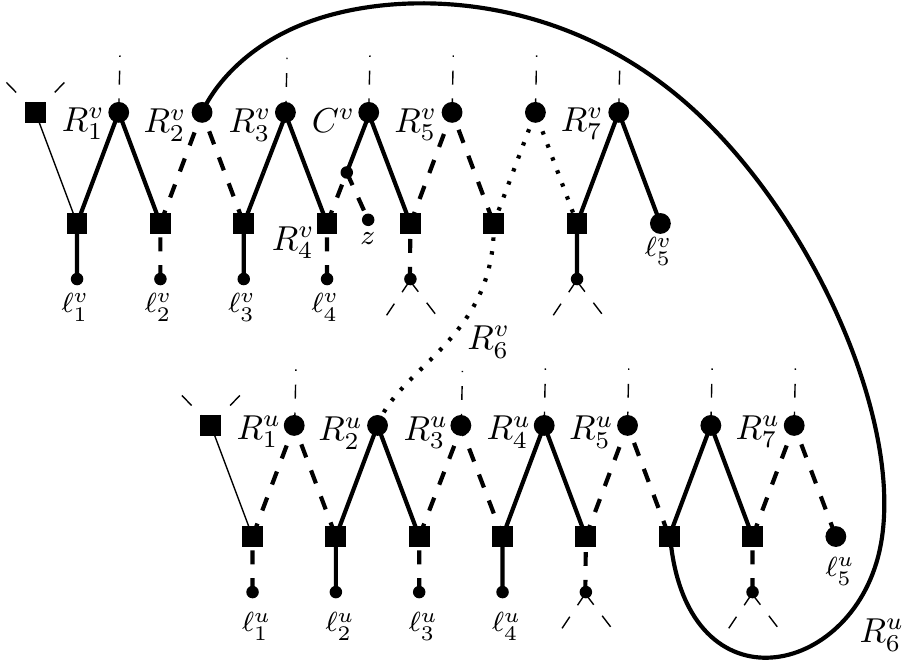}\label{subfig:gadgets_cherry_cover_leaf}}
        \subfloat[]{\includegraphics[height=4cm]{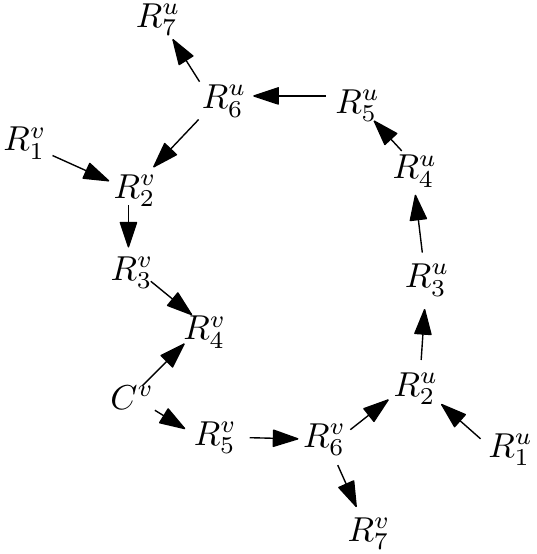}\label{subfig:gadgets_cherry_cover_leaf_graph}}
    \caption{Cherry cover of~$\Gad(v)$ and~$\Gad(u)$. In (a), the unique cherry cover of the principal part of~$\Gad(v)$ and~$\Gad(u)$ is displayed, in (b), the cherry cover auxiliary graph of (a) is given. In (c), the leaf~$z \notin X$ is added to the principal part of~$\Gad(v)$, and one possible cherry cover of the same part of the network is given. And in (d), the cherry cover auxiliary graph of (c) is given.}
    \label{fig:reduction_cherry_cover}
\end{figure}

\begin{lemma}
\label{lem:n_fence_cycle}
Let $N$ be a tree-based network and suppose that for two N-fences 
$\hat{N_u} := (a^u_1,a^u_2,\ldots, a^u_{k_u})$ and 
$\hat{N_v} := (a^v_1,a^v_2,\ldots, a^v_{k_v})$ of length at least $3$, there exist directed paths in $N$ from $\head(a^u_h)$ to $\tail(a^v_i)$ and from $\head(a^v_j)$ to $\tail(a^u_k)$, for even $h,i,j,k$ with $k < h$ and $i < j$. 
Then every cherry cover auxiliary graph of~$N$ contains a cycle.
\end{lemma}

\markj{In order to remove all possible cycles from a possible cherry cover, it is therefore necessary to disrupt the principal part of either $\Gad(u)$ or $\Gad(v)$, for any edge~$uv$ in $G$.}

\begin{lemma}\label{lem:LORDimpliesVC}
    Let~$G$ be a 3-regular graph and let~$N_G$ be the network obtained by the reduction above. Suppose that~$A$ is a set of arcs of~$N_G$, for which adding leaves to every arc in~$A$ results in an orchard network.  
    For every edge~$uv\in E(G)$, there exists an arc~$a\in A$ that is an arc of the principal part of~$\Gad(u)$ or~$\Gad(v)$.
\end{lemma}
  


\markj{To complete the proof of the validity of the reduction, we show that in order to make $N_G$ orchard by leaf additions, it is sufficient (and necessary) to add a leaf~$z^v$ to an appropriate arc of~$\Gad(v)$ for every $v$ in a vertex cover~$V_{sol}$ of $G$ (see \Cref{subfig:gadgets_cherry_cover_leaf}). 
    The key idea is that this splits the principal part of $\Gad(v)$ from an N-fence into an N-fence and an M-fence, and this allows us to avoid the cycle in the cherry cover auxiliary graph (see \Cref{subfig:gadgets_cherry_cover_leaf_graph}).}

\begin{lemma}\label{lem:VCiffLORD}
    Let~$G$ be a 3-regular graph and let~$N_G$ be the network obtained by the reduction described above. Then $G$ has a minimum vertex cover of size~$k$ if and only if~$L_{\cOr}(N_G) = k$.
\end{lemma}

\begin{theorem}\label{thm:L_Or=Hard}
    Let~$N$ be a network. The decision problem \dpLD{} is NP-complete. Computing $L_\cOr(N)$ is NP-hard. 
\end{theorem}

\begin{proof}
    Suppose we are given a set of arcs~$A_{sol}$ of~$N_G$ of size at most~$k$.
    Upon adding leaves to every arc in~$A_{sol}$, we may check that the resulting network is orchard in polynomial time (see Section 6 of \cite{janssen2021cherry}).
    This implies that~\dpLD{} is in NP. The reduction from \dpVC{} to \dpLD{} outlined at the start of the section takes polynomial time, since we add a finite number of vertices and arcs for every vertex in the \dpVC{} instance. The NP-completeness of \dpLD{} follows from the equivalence of the two problems shown by \Cref{lem:VCiffLORD}. The optimization problem of \dpLD{}, i.e., the one of computing~$L_\cOr(N)$ is therefore NP-hard.
\end{proof}

\section{Upper Bound}\label{sec:Bound}

In the previous section we showed that computing~$L_\cOr(N)$ is NP-hard.
Here, we provide a sharp upper bound for~$L_{\cOr}(N)$.
We call a reticulation \emph{highest} if it has no reticulation ancestors.

\begin{lemma}
\label{lem:RetHasLeafSibling}
    Let~$N$ be a network. Suppose there is a highest reticulation~$r$ such that all other reticulations have a leaf sibling. Then~$N$ is orchard.
\end{lemma}
\begin{proof}
    We prove the lemma by induction on the number of reticulations~$k$. For the base case, observe that a network with one reticulation is tree-child since it has no omnians. A tree-child network is orchard~\cite{janssen2021cherry}, and so this network must be orchard.

    Suppose now that we have proven the lemma for all networks with fewer than~$k$ reticulations, where~$k>1$. 
    Let~$N$ be a network with reticulation set~$R$ where~$|R| = k$, and suppose there exists a highest reticulation~$r$ in~$N$ such that all other reticulations have a leaf sibling.
    Let~$r$ denote the highest reticulation as specified in the statement of the lemma.
    Choose a lowest reticulation~$r'\in R\setminus \{r\}$. By assumption,~$r'$ has a leaf sibling~$c$.
    Every vertex below~$r'$ must be tree vertices and leaves. Reduce cherries until the child~$x$ of~$r'$ is a leaf. 
    Then~$(x,c)$ is a reticulated cherry; the network~$N'$ obtained by reducing this reticulated cherry has $k-1$ reticulations and has a highest reticulation~$r$ such that all other reticulations have a leaf sibling.
    By induction hypothesis,~$N'$ must be orchard. Since a sequence of cherry reductions can be applied to~$N$ to obtain~$N'$, the network~$N$ must also be orchard.
\end{proof}

\begin{theorem}
\label{thm:LORUpperBound}
    Let $N$ be a network, and let~$r(N)$ denote the number of reticulations. Then~$L_\cOr(N)=0$ if~$N$ is a tree, and otherwise, $L_{\cOr}(N) \leq r(N) - 1$, where the bound is sharp.
\end{theorem}
\begin{proof}
    If~$N$ is a tree, then it is orchard, and so~$L_\cOr(N) = 0$.
    So suppose~$r(N)>0$. Let $r$ be a highest reticulation of~$N$, and for every other reticulation, arbitrarily choose one incoming reticulation arc.
    Add a leaf to each of these reticulation arcs. By \Cref{lem:RetHasLeafSibling}, the resulting network must be orchard.
    We have added a leaf for all but one reticulation in~$N$. It follows that~$L_\cOr(N) \le r(N) - 1$.
    The network in \Cref{fig:Lor_bound} shows that this upper bound is sharp.
\end{proof}

\begin{figure}
    \centering
    \includegraphics[width=0.22\columnwidth]{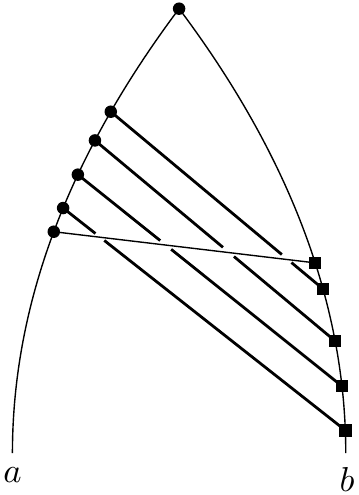}
    \caption{A network~$N$ on two leaves~$\{a,b\}$ with~$r(N) = 5$ reticulations. Observe that $L_{\cOr}(N) = r(N) - 1 = 4$, since the highest reticulation cannot be reduced by cherry picking unless the reticulations below it are first reduced. For each non-highest reticulation, we must add a leaf to one of its incoming arcs to reduce it, which leads to $L_{\cOr}(N) = 4$. Note that this construction can be extended for any~$k$ reticulations. 
    }
    \label{fig:Lor_bound}
\end{figure}

\section{MILP Formulation}
\label{sec:ILP}

To model the problem of computing the leaf addition proximity measure as a MILP, we reformulate the measure in terms of non-temporal labellings.

\subsection{Vertical Arcs into Reticulations}\label{subsec:ProximityMeasures_H}

By \Cref{thm:OrchIFFHori}, every orchard network can be viewed as a network 
with a base tree where each of the linking arcs are horizontal. Recall that in 
terms of non-temporal labellings, this means that there exists a labelling wherein every reticulation has exactly one 
incoming reticulation arc that is horizontal. Following this definition, we introduce
a second orchard proximity measure. Given a non-temporal labelling for a network~$N$, let us use \emph{inrets} to refer to reticulations of~$N$ with only vertical incoming arcs. 
Let~$V_\cOr(N)$ denote the minimum number of inrets over all possible non-temporal
labellings. 

\begin{observation}\label{obs:NoInret=HGT}
    Let~$N$ be a network. A network admits an HGT-consistent labelling if and only if~$V_{\cOr}(N) = 0$. In other words, a network is orchard if and only if~$V_{\cOr}(N) = 0$. 
\end{observation}

In particular, we show a stronger result that equates the two proximity measures.

\begin{lemma}\label{lem:L=H}
Let~$N$ be a network. Then $L_{\cOr}(N) = V_{\cOr}(N)$.
\end{lemma}
\begin{proof}

Suppose first that we have a network~$N$ with some non-temporal labelling~$t : V(N)\rightarrow \mathbb{R}$ which gives
rise to~$h$ inrets. For every inret~$r$ with parents~$u$ and~$v$, we add a 
leaf~$x$ to the arc~$ur$ (this addition is done without loss of generality; the argument also follows by adding the leaf to~$vr$). Since~$r$ is
an inret, we must have~$t(u)< t(r)$ and~$t(v)<t(r)$. Letting~$p_x$ denote the parent
of~$x$, we label~$t(p_x) := t(r)$ and~$t(x) := t(p_x) + 1$. This ensures that the 
extension of the map~$t$ that includes~$x$ and~$p_x$ is a non-temporal labelling for~$N+x$. 
Observe that~$r$ is no longer an inret in~$N+x$, since the arc~$p_xr$ is horizontal. 
Therefore, a leaf addition to an incoming arc of an inret can reduce the number of inrets by exactly one. By repeating this procedure for every inret, it follows that~$L_{\cOr}(N) \le V_{\cOr}(N)$.
\medskip

To show the other direction, suppose we can add~$\ell$ leaves to~$N$ to make it orchard. By \Cref{thm:OnlyAddToRetArcs}, we may assume all such leaves are added to reticulation arcs in the set~$\{e_1,\ldots, e_\ell\}$. 
The resulting network~$N'$ has an HGT-consistent labelling $t:V(N') \rightarrow \mathbb{R}$ by \Cref{thm:OrchIFFHori}. 

We claim that the labelling~$t|_{V(N)}$ restricted to~$N$ is a non-temporal labelling, and that under~$t|_{V(N)}$, the number of inrets is at most~$\ell$. 
Suppose that a leaf~$x_i$ was added to the reticulation arc~$e_i = u_ir_i$. Let~$p_i$ denote the parent of~$x_i$ in the network~$N'$.
By definition of HGT-consistent labellings, we must have that~$t(u_i) < t(r_i)$, since~$u_ip_ir_i$ is a path in~$N'$. 
Therefore, restricting the labelling to the network obtained from~$N'$ by removing the leaf~$x_i$ is non-temporal.
Furthermore, if~$v_i$ is the parent of~$r_i$ that is not~$u_i$, we have that one of~$v_ir_i$ or~$p_ir_i$ must be horizontal in~$N'$.
If~$v_ir_i$ was horizontal, then~$r_i$ still has a horizontal incoming arc upon removing~$x_i$, and the number of inrets does not change. On the other hand, if~$p_ir_i$ was horizontal, then~$v_ir_i$ must have been a vertical arc. Upon deleting~$x_i$, the reticulation~$r_i$ becomes an inret as its other incoming arc~$u_ir_i$ is also vertical.
Since leaf deletions are local operations, deleting a leaf increases the number of inrets by at most one. 
By repeating this for each reticulation arc~$e_i$ for~$i\in[\ell]$, it follows that $N$ contains at most~$\ell$ inrets, and therefore~$V_\cOr(N) \le L_\cOr(N)$.
\end{proof}

\subsection{MILP Formulation}
By \Cref{lem:L=H} we have that $L_{\cOr}(N) = V_{\cOr}(N)$. In this section, we introduce a MILP formulation to obtain $V_{\cOr}(N)$, and therefore also $L_{\cOr}(N)$. This is done by searching for a non-temporal labelling of networks in which the number of vertical arcs is minimized. 

Let $N$ be a given network with vertex set~$V$ and arc set~$A$. Let~$R$ denote the set of reticulations of~$N$. 
We define the decision variable $l_v$ to be the non-temporal label of the vertex $v \in V$. A tree arc and a vertical linking arc $uv$ have the property that $l_u < l_v$. We define $x_a$ to be one if arc $a \in A$ is vertical and zero otherwise. 
We define $h_v$ to be one if $v \in R$ is a reticulation with only incoming vertical arcs and zero otherwise.  
Let~$v\in V$ be a vertex of~$N$. In what follows, let $P_v \subset V$ be the set of parent nodes of $v$, $C_v \subset V$ the set of children nodes of $v$, and $X$ the set of leaves. Let~$\rho$ be the root of~$N$. Then, the MILP formulation is as follows: 
\begin{align}
    \min_{x, h, l} \quad   & \sum_{v \in R} h_v \nonumber \\
    \text{s.t.} \quad   & \sum_{u \in P_v} x_{uv} - 1 \leq h_v & \forall v \in R \label{eq:h_cons}\\
                        & \sum_{v \in C_u} x_{uv} \geq 1  & \forall u \in V \setminus X \label{eq:outgoing_arcs} \\
                        & \sum_{u \in P_v} x_{uv} \geq 1  & \forall v \in V \setminus \{\rho\} \label{eq:inrets_arcs} \\
                        & l_u \leq l_v  & \forall uv \in A \label{eq:temp_label} \\
                        & l_u \leq l_v - 1 & \forall v \in V \setminus R, \forall u \in P_v \label{eq:ta1_label}\\
                        & l_u \leq l_v - 1 + |V|(1 - x_{uv}) & \forall v \in R, \forall u \in P_v  \label{eq:ta2_label}\\
                        & l_u \geq l_v - |V| x_{uv} & \forall v \in R, \forall u \in P_v \label{eq:ha_label} \\
                        & x_a \in \{0, 1\} & \forall a \in A \nonumber \\
                        & h_v \in \{0, 1\} & \forall v \in R \nonumber \\
                        & l_v \in \mathbb{R_+} & \forall v \in V \nonumber
\end{align}

With constraint \eqref{eq:h_cons}, $h_v$ becomes one if all incoming arcs of reticulation $v$ are vertical. With \eqref{eq:outgoing_arcs} we have that all vertices must have at least one outgoing vertical arc. Then, \eqref{eq:inrets_arcs} guarantees that each reticulation has at least one incoming vertical arc. Constraint \eqref{eq:temp_label} creates the non-temporal labelling in the network, where with \eqref{eq:ta1_label} the label of $u$ is strictly smaller than that of $v$ if $v$ is not a reticulation. Then, \eqref{eq:ta2_label} sets $x_{uv}$ to one if $uv$ is vertical, for all reticulation vertices $v$. Finally, with \eqref{eq:ha_label} the labels of $u$ and $v$ become equal if $x_{uv}$ is zero.

\subsection{Experimental Results}
\label{sec:experiments}
In this section, we apply the MILP described in the previous section to a set of real binary networks, and we perform a runtime analysis over simulated networks. The code for these experiments is written in Python and is available at \url{https://github.com/estherjulien/OrchardProximity}. All experiments ran on an Intel Core i7 CPU @  1.8 GHz with 16 GB RAM. For solving the MILP problems, we use the open-source solver SCIP \cite{BestuzhevaEtal2021OO}.

The real data set consists of different binary networks found in a number of papers, collected on \url{http://phylnet.univ-mlv.fr/recophync/networkDraw.php}. These networks have a leaf set of size up to 39 and a number of reticulations up to 9, with one outlier that has 32 reticulations. All the \leo{binary} instances completed within one second (at most 0.072 seconds). Based on the results, we observe that only two out of the 22 \leo{binary} networks have a value of $L_{\cOr}(N) > 0$, thus, that only two are non-orchard.

The simulated data is generated using the birth-hybridization network generator of~\cite{zhang2018bayesian}. This generator has two user-defined parameters: $\lambda$, which regulates the speciation rate, and $\nu$, which regulates the hybridization rate. Following~\cite{janssen2021comparing} we set $\lambda = 1$ and we sampled $\nu\in[0.0001, 0.4]$ uniformly at random. We generated an instance group of size $50$ for each pair of values $(L, R)$, with the number of leaves $L \in \{20, 50, 100, 150, 200\}$ and the number of reticulations $R \in \{5, 10, 20, 30, 40, 50, 100, 200\}$. In our implementation, we only defined variables $x_a$ for incoming reticulation arcs. Therefore, the binary variables only depend on the number of reticulations in the network. In \Cref{tab:simu_results}, the runtime and $L_{\cOr}(N)$ value results for the simulated instances are shown against the reticulation number of the networks. The time limit was set to one hour. We can observe from these results that for networks up to 50 reticulations, almost all instances are solved to optimality within a second. Then for $R=100, 200$ the runtime increases, but \leo{instances are} often still solved within reasonable time. 

\begin{table}
\small
    \centering
        \caption{Results of simulated networks, with $R$ the number of reticulations in the network, \textit{Compl. (\%)} the percentage of instances that were solved within the one hour time limit, \textit{Runtime compl.} gives the 5\%, 50\%, and 95\% quantile for the runtimes of the completed instances and, finally, $L_{\cOr}(N)$ gives the different quantile percentages for the optimal objective value of the completed instances.}
    \label{tab:simu_results}
\begin{tabular}{rrrrrrrrrrr}
\toprule
  \multicolumn{1}{c}{$R$} & & \multicolumn{1}{c}{Compl. (\%)} & &\multicolumn{3}{c}{Runtime compl. (s)}  &  & \multicolumn{3}{c}{$L_{\cOr}(N)$}\\ \hhline{-~-~---~---}
   & & & & 5\% & 50\% &  95\% & &  5\% &  50\% &  95\% \\ 
\midrule
  5 & &              100.0 & & 0.002 &   0.008 &     0.018 &  &      0 &         0 &         0 \\
 10 & &               99.2 &  & 0.004 &   0.011 &     0.023 &  &      0 &         0 &         2 \\
 20 & &               98.8 &  & 0.008 &   0.016 &     0.051 &  &      0 &         0 &         4 \\
 30 & &               98.0 &  & 0.013 &   0.023 &     0.089 &  &      0 &         1 &         7 \\
 40 & &               97.2 &  & 0.018 &   0.038 &     0.189 &  &      0 &         1 &        10 \\
 50 & &               97.6 &  & 0.022 &   0.049 &     0.632 &  &      0 &         2 &        13 \\
100 & &               75.6 &  & 0.058 &   0.233 &   109.395 &  &      2 &         8 &        28 \\
200 & &               61.6 &  & 0.614 &  15.783 &  3,078.631 &  &     14 &        27 &        56 \\
\bottomrule
\end{tabular}
\end{table}

\section{Discussion}
\label{sec:discussion}

In this paper we investigated the minimum number of leaf additions needed to make a network orchard, as a way to measure the extent to which a network deviates from being orchard.
We showed that computing this measure is NP-hard (\Cref{thm:L_Or=Hard}), and give a sharp upper bound by the number of reticulations minus one (\Cref{thm:LORUpperBound}). 
The measure was reformulated to one in terms of minimizing the number of inrets over all possible non-temporal labellings. 
In Section~\ref{sec:ILP} we use this reformulation to model the problem of computing the leaf addition measure as a MILP.
Experimental results show that real-world data instances were solved within a second and the formulation worked well also over synthetic instances, being able to solve \leo{almost all instances} up to 50 reticulations and 200 leaves within one second. For bigger instances the runtime however increased.

\leo{Our NP-hardness result is interesting when compared it to the computational complexity of the corresponding problem for different network classes.}
The problem of finding the minimum number of leaves to add to make a network tree-based can be solved in polynomial time~\cite{francis2018new} (\Cref{lem:L_TB=polynomial}) \leo{and we showed that the same is true for the class of tree-child networks (\Cref{lem:L_TCisPoly}).} 
Interestingly, the class of tree-child networks \leo{is} contained in the class of orchard networks~\cite{janssen2021cherry} which \leo{is} in turn contained in the class of tree-based networks~\cite{huber2019rooting}.
The reason for such an NP-hardness sandwich can \leo{perhaps} be attributed to the lack of forbidden shapes.
Leaf additions to obtain tree-child or tree-based networks target certain forbidden shapes in the network. In the case of tree-child networks, we add a leaf to exactly one outgoing arc of \leo{each} omnian; for tree-based networks, we add a leaf to exactly one arc of \leo{each} W-fence.
\leo{The problem} of finding a \leo{characterization of orchard networks in terms of} (local) forbidden shapes
has been elusive thus far~\cite{janssen2021cherry} - perhaps the NP-hardness result for the orchard variant of the problem indicates \leo{that} finding \leo{such a characterization} for orchard networks \leo{may not be possible}.



One can also consider the leaf addition problem for non-binary networks.
Non-binary networks generalize the networks considered in this paper by allowing vertices to have varying indegrees and outdegrees.
This generalized problem remains NP-complete since the \leo{binary version} is a specific case. \leo{It could be interesting to try to find an MILP formulation for the nonbinary version.}


Another natural research direction is to consider different proximity measures.
One that may be of particular interest is a proximity measure based on arc deletions. That is, what is the minimum number of reticulate arc deletions needed to make a network orchard? 
Susanna showed that \leo{this measure is incomparable to} the leaf addition proximity measure~\cite{susanna2022making}, yet it is not known if it is also NP-hard to compute. 


\bibliography{z_bib}

\appendix

\section{Appendix}
\label{sec:appendix}

\begin{lemman}[\ref{lem:L_TC=Omnian}]
  \textit{Let~$N$ be a network. Then $L_{\cTC}(N)$ is equal to the number of omnians. Moreover,~$N$ can be made tree-child by adding a leaf to exactly one outgoing arc of each omnian.}
\end{lemman}
\begin{proof}
	By definition, a network is tree-child if and only if it contains no 
	omnians. We show that every leaf addition can result in a network with one 
	omnian fewer than that of the original network.
	Let~$uv$ be an arc where~$u$ is an omnian. Add a leaf~$x$ to~$uv$. In the 
	resulting network,~$u$ has a child (the parent of~$x$) that is a tree 
	vertex, and it is no longer an omnian. The newly added tree vertex has a 
	leaf child~$x$; the parent-child combinations remain unchanged for the rest 
	of the network, so at most one omnian (in this case~$u$) can be removed per 
	leaf 
	addition. It follows that~$L_{\cTC}(N)$ is at least the number of omnians 
	in~$N$. By targeting arcs with omnian tails, we can remove at least one 
	omnian per every leaf addition, so that~$L_{\cTC}(N)$ is at most the number 
	of omnians in~$N$. Therefore,~$L_{\cTC}(N)$ is exactly the number of 
	omnians in~$N$.
\end{proof}

\begin{lemman}[\ref{lem:L_TCisPoly}]
    \textit{Let~$N$ be a network. Then~$L_{\cTC}(N)$ can be computed in~$O(|N|)$ time.}
\end{lemman}
\begin{proof}
    We first show that the number of omnians of~$N$ can be computed in~$O(|N|)$ 
    time, by checking, for each vertex, the indegrees of its children. A vertex 
    is an omnian if and only if all of its children are of indegree-2. Since 
    the degree of every vertex is at most 3, each search within the for loop 
    takes constant time. The for loop iterates over the vertex set which is of size~$O(N)$. 
    By \Cref{lem:L_TC=Omnian}, since~$L_{\cTC}(N)$ is the number of omnians in~$N$, we can compute~$L_{\cTC}(N)$ in~$O(|N|)$ time.
\end{proof}

It has been shown already that~$L_{\cTB}(N)$ can be computed in~$O(|N|^{3/2})$ time 
where~$|N|$ is the number of vertices in~$N$~\cite{francis2018new}. We show that this can in fact be computed in~$O(|N|)$ time.

\begin{lemman}[\ref{lem:L_TB=W-fences}]
    \textit{Let $N$ be a network. Then $L_{\cTB}(N)$ is equal to the number of $W$-fences. Moreover, $N$ can be made tree-based by adding a leaf to any arc in each $W$-fence in $N$.}
\end{lemman}
\begin{proof}
    By \Cref{lem:TB=NoW}, a network is tree-based if and only if it 
	contains no W-fences. We show that every leaf addition can result in a 
	network with one W-fence fewer than that of the original network.
	Suppose that~$N$ contains at least one 
	W-fence. 
	Otherwise we may conclude that the network is tree-based by 
	\Cref{lem:TB=NoW}. 
	Let~$(a_1,a_2,\ldots, a_k)$ be a W-fence in~$N$ where $a_i = u_iv_i$ 
	for~$i\in[k]$, and add a leaf~$x$ to~$a_1$; let~$p_x$ be the tree vertex 
	parent of~$x$. In the resulting network, the arcs 
	in~$\{u_1p_x,p_xv_1,p_xx, a_2, a_3,a_4,\ldots, a_k\}$ are decomposed into 
	their unique maximal zig-zag trails (\Cref{thm:UniqueMaxZigZag}) as two 
	N-fences~$(u_1p_x)$ and~$(a_k,a_{k-1},\ldots,a_3,a_2,p_xv_1,p_xx)$. All 
	other arcs remain in the same maximal zig-zag trails as that of~$N$. 
	Therefore the number of W-fences has gone down by exactly one. This can be 
	repeated for all W-fences in the network; it follows that~$L_{\cTB}(N)$ is 
	the number of W-fences in~$N$.

    A quick check shows that adding a leaf to any arc in the W-fence decomposes the W-fence into two N-fences.
\end{proof}

\begin{lemman}[\ref{lem:L_TB=polynomial}]
	\textit{Let~$N$ be a network. Then $L_\cTB(N)$ can be computed in~$O(|N|)$ time.}
\end{lemman}
\begin{proof}
   Finding the maximal zig-zag decomposition takes $O(|N|)$ time (Proposition 5.1 of \cite{hayamizu2021structure}). Counting the number of $W$-fences in the decomposition gives~$L_\cTB(N)$ by \Cref{lem:L_TB=W-fences}.
\end{proof}

\begin{obsn}[\ref{obs:N_GisTB}]
    \textit{Let~$G$ be a 3-regular graph and let~$N_G$ be the network obtained by the reduction. Then~$N_G$ is tree-based.}
\end{obsn}
\begin{proof}
    It is easy to check that the arcs of~$N_G$ are decomposed into M-fences and N-fences (the principal part of each gadget~$\Gad(v)$ is an N-fence; each arc leaving the principal part of a gadget~$\Gad(v)$ is an N-fence of length $1$; the remaining arcs decompose into M-fences of length $2$). By \Cref{lem:TB=NoW},~$N_G$ must be tree-based. 
\end{proof}

\begin{lemman}[\ref{lem:LORDimpliesVC}]
    \textit{Let~$G$ be a 3-regular graph and let~$N_G$ be the network obtained by the reduction above. Suppose that~$A$ is a set of arcs of~$N_G$, for which adding leaves to every arc in~$A$ results in an orchard network.  
    For every edge~$uv\in E(G)$, there exists an arc~$a\in A$ that is an arc of the principal part of~$\Gad(u)$ or~$\Gad(v)$.}
\end{lemman}
\begin{proof}
    We prove this lemma by contraposition. Let us assume that there is an edge~$uv\in E(G)$, such that no arcs of the principal part of~$\Gad(u)$ or~$\Gad(v)$ are in~$A$.
    We shall show that the network obtained by adding leaves to all~$a \in A$ in $N_G$ -- which we denote~$N_G+A$ -- is not orchard.
  
    From \Cref{thm:AcyclicCherryCover} we know that $N_G$ is orchard if and only if~$N_G$ has an acyclic cherry cover.
    We show here that~$N_G+A$ will not have an acyclic cherry cover, thereby showing that~$N_G+A$ is not orchard.

    As no arcs were added to the principal part of~$\Gad(u)$ or~$\Gad(v)$, these principal parts remain N-fences in $N_G+A$. Furthermore by construction $N$ has an arc from some~$\head(a_h^v)$ to~$\tail(a_i^u)$ for even $h \geq 10$ and even $i \leq 6$, and so~$N_G+A$ has a path from~$\head(a_h^v)$ to~$\tail(a_i^u)$. Similarly ~$N_G+A$ has a path from~$\head(a_j^u)$ to~$\tail(a_k^v)$ for some even $j \geq 10$ and $h\leq 6$. Then \Cref{lem:n_fence_cycle} implies that the auxiliary graph of any  cherry cover of  $N_G+A$ contains a cycle. By \Cref{thm:AcyclicCherryCover}, we have that~$N_G+A$ is not orchard.  
\end{proof}


\begin{lemman}[\ref{lem:n_fence_unique}]
    \textit{Let~$N$ be a tree-based network, and let~$\hat{N_1},\hat{N_2},\ldots, \hat{N_n}$ denote the N-fences of~$N$ of length at least 3. 
    Then every cherry cover of~$N$ contains the reticulated cherry shapes~$\{(\head(a^i_{2j-1})c^i_{2j-1}), a^i_{2j},a^i_{2j+1}\}$ for~$i\in[n]$ and~$j\in\left[\frac{k_i-1}{2}\right]$.}
\end{lemman}
\begin{proof}
    Let~$\hat{N_i} = (a^i_1,a^i_2,\ldots, a^i_{k_i})$ be an N-fence of length~$k_i\ge 3$.
    Observe that in every cherry cover, exactly one incoming arc of every reticulation is covered by a reticulated cherry shape as a middle arc (since the network is binary; for non-binary networks, this is not true in general~\cite{van2021unifying}).
    Since~$\head(a^i_1)$ is a reticulation, one of~$a^i_1$ or~$a^i_2$ must be in a reticulated cherry shape as a middle arc. But~$\tail(a^i_1)$ is a reticulation; therefore,~$a^i_2$ must be in a middle arc of a reticulated cherry shape.
    The other two arcs of the same reticulated cherry shapes are then fixed to be~$\head(a^i_1)c^i_1$ and~$a^i_3$. 
    Repeating this argument for the reticulations~$\head(a^i_{2j+1})$ for~$j\in\left[\frac{k_i-1}{2}\right]$ gives the required claim for the N-fence~$\hat{N_i}$; further repeating this argument for every N-fence gives the required claim.
\end{proof}

\begin{lemman}[\ref{lem:n_fence_cycle}]
\textit{Let $N$ be a tree-based network and suppose that for two N-fences 
$\hat{N_u} := (a^u_1,a^u_2,\ldots, a^u_{k_u})$ and 
$\hat{N_v} := (a^v_1,a^v_2,\ldots, a^v_{k_v})$ of length at least $3$, there exist directed paths in $N$ from $\head(a^u_h)$ to $\tail(a^v_i)$ and from $\head(a^v_j)$ to $\tail(a^u_k)$, for even $h,i,j,k$ with $k < h$ and $i < j$. 
Then every cherry cover auxiliary graph of~$N$ contains a cycle.}
\end{lemman}
\begin{proof}
    Let us again denote the reticulated cherry shape that contains~$a^u_h$ and~$a^u_{h+1}$ by~$R^u_{h/2}$, and similarly for~$R^v_{i/2}$,~$R^v_{j/2}$,and~$R^u_{k/2}$.
    By \Cref{lem:n_fence_unique}, all of ~$R^u_{h/2}$,~$R^v_{i/2}$,~$R^v_{j/2}$,~$R^u_{k/2}$ appear in the cherry cover auxiliary graph. Moreover~$R^u_{k/2}$ is above~$R^u_{h/2}$, and~$R^v_{i/2}$ is above~$R^v_{j/2}$.
    Now observe that for any consecutive arcs on the path from $\head(a^u_h)$ to $\tail(a^v_i)$, either they are part of the same reticulated cherry shape in the cherry cover, or they are part of different cherry shapes with one cherry shape directly above the other. This implies that there is a path from~$R^u_{h/2}$ to~$R^v_{i/2}$ in the cherry cover auxiliary graph. A similar argument shows that there is a path from~$R^v_{j/2}$ to~$R^u_{k/2}$. 
    But then we have that~$R^u_{h/2}$ is above~$R^v_{i/2}$, which is above~$R^v_{j/2}$, which is above~$R^u_{k/2}$, which is above~$R^u_{h/2}$ and we have a cycle.
\end{proof}

\begin{lemman}[\ref{lem:VCiffLORD}]
    \textit{Let~$G$ be a 3-regular graph and let~$N_G$ be the network obtained by the reduction described above. Then $G$ has a minimum vertex cover of size~$k$ if and only if~$L_{\cOr}(N_G) = k$.}
\end{lemman}
\begin{proof}
    Suppose first that~$V_{sol}$ is a vertex cover of~$G$ with at most~$k$ vertices.
    We shall show that 
    adding a leaf to an arc of the principal part of each~$\Gad(v)$ for~$v \in V_{sol}$ makes $N_G$ orchard. 
    This will show that the minimum vertex cover of~$G$ is at least $L_{\cOr}(N_G)$. 
    In the remainder of this proof, we will refer to vertices and arcs of~$N_G$ as introduced above in the reduction.
    
    
    For every~$v \in V_{sol}$, we add a leaf~$z^v$ to the arc~$w_4^v r_4^v$ of~$\Gad(v)$ (see \Cref{subfig:gadgets_cherry_cover_leaf}). Let~$q^v$ be the parent of~$z^v$.
    The key idea is that this splits the principal part of $\Gad(v)$ from an N-fence into an N-fence and an M-fence, and this allows us to avoid the cycle in the cherry cover auxiliary graph (see \Cref{subfig:gadgets_cherry_cover_leaf_graph}).
    
    Let us call the new network~$M$.
    To formally show that $M$ is orchard, we give an HGT-consistent labelling~$t:V(M)\rightarrow\R$.
    
    Begin by setting $t(\rho) = 0$, and for any vertex in $s_1,\dots, s_{g-1}$ or $\rho^v, m_4^v,\dots, m_2^v$ for any $v$ in $V(G)$, let this vertex have label equal to the label of its parent plus~$1$. 
    Let $h$ be the maximum value assigned to a vertex so far, and now adjust $t$ by subtracting $(h+1)$ from each label. Thus, we may now assume that all vertices in $\rho, s_1,\dots, s_{g-1}$ or $m_4^v,\dots, m_2^v$ for any $v$ in $V(G)$ have label $\leq -1$. Now set  $t(m_1^v) = 0$ and $t(r_0^v) = 0$, for each $v$ in $V(G)$.
    
    It is easy to see that so far $t$ satisfies the properties of an HGT-consistent labelling. It remains to label the vertices in the principal part of each gadget~$\Gad(v)$, and the leaves of each gadget, and the new vertices $q^v$ and $z^v$ for $v \in~V_{sol}$. We do this as follows.
    
    For $v \in~V_{sol}$, set $t(r_1^v)=t(w_1^v) = 12, t(r_2^v)=t(w_2^v) = 13, t(r_3^v)=t(w_3^v) = 14$, and $t(r_4^v) = t(q^v) = 15$. 
    Set $t(w_4^v) = 1, t(r_5^v) = t(w_5^v)=2, t(r_6^v) = t(w_6^v) = 3$, and $t(r_7^v) = t(w_7^v) = 4$.
    
    For $v \notin~V_{sol}$, set $t(r_1^v) = t(w_1^v) = 5$, and $t(r_i^v)=t(w_i^v) = i+4$ for every $i$ up to $t(r_7^v) = t(w_7^i) = 11$.
    
    Finally, for each leaf $\ell$ with parent $p$ set $t(\ell) = t(p) + 1$.
    
    It remains to observe that $t$ is a non-temporal labelling of $M$ and for every reticulation $r$ in $M$, $r$ has exactly one parent $p$ with $t(p)=t(r)$. Thus $t$ is an HGT-consistent labelling of $M$, and it follows from~\cref{thm:OrchIFFHori} that $M$ is orchard.  
    
    \medskip
    
    Suppose now that we have a set of arcs~$A_{sol}$ of~$N_G$ of size at most~$k$, such that adding leaves to the arcs in~$A_{sol}$ makes~$N_G$ orchard.
    By \Cref{lem:LORDimpliesVC}, for every edge~$uv\in E(G)$, there exists an arc~$a\in A_{sol}$ that is an arc of the principal part of~$\Gad(u)$ or~$\Gad(v)$. 
    It follows immediately that the set~$\{v\in V(G): \text{$A_{sol}$ contains an arc of the principal part of~$\Gad(v)$}\}$ is a vertex cover of~$G$. 
    Since this is true for any such set of arcs~$A_{sol}$, it follows that if there is such an~$A_{sol}$ of size at most~$k$, then there must exist a vertex cover of~$G$ of size at most~$k$. 
\end{proof}

\end{document}